\documentclass[12pt]{amsart}
\usepackage{amsmath,amsthm,amsfonts,amssymb}
\usepackage{graphicx}
\usepackage{mathrsfs}
\usepackage{enumitem}
\usepackage{etoolbox}
\usepackage{xcolor}
\apptocmd{\sloppy}{\hbadness 10000\relax}{}{}
\apptocmd{\sloppy}{\vbadness 10000\relax}{}{}
\usepackage[pdfpagelabels]{hyperref}
\usepackage[letterpaper,margin=1.1in]{geometry}

\numberwithin{equation}{section}
\theoremstyle{plain}

\newtheorem{theorem}{Theorem}[section]

\newtheorem{corollary}[theorem]{Corollary}

\newtheorem{lemma}[theorem]{Lemma}

\theoremstyle{definition}
\newtheorem{remark}[theorem]{Remark}
\newtheorem{definition}[theorem]{Definition}

\newcommand{\norm}[1]{\left\lVert#1\right\rVert}

\def\XXint#1#2#3{{\setbox0=\hbox{$#1{#2#3}{\int}$ }
\vcenter{\hbox{$#2#3$ }}\kern-.6\wd0}}

\newcommand{\diam}{\mathop\mathrm{diam}\nolimits}
\newcommand{\dist}{\mathop\mathrm{dist}\nolimits}

\begin{document}

\title{Conditions for Eliminating Cusps in One-phase free-boundary problems with degeneracy}

\author{Sean McCurdy}
\keywords{Free-boundary problems, Alt-Caffarelli functional, Stokes wave, partial regularity}

\address{Department of Mathematics\\ National Taiwan Normal University\\ Taipei, Taiwan}
\email{smccurdy@ntnu.edu.tw}

\begin{abstract} In this paper, we continue the study local minimizers of a degenerate version of the Alt-Caffarelli functional. Specifically, we consider local minimizers of the functional $J_{Q}(u, \Omega):= \int_{\Omega} |\nabla u|^2 + Q(x)^2\chi_{\{u>0\}}dx$ where $Q(x) = dist(x, \Gamma)^{\gamma}$ for $\gamma>0$ and $\Gamma$ a  submanifold of dimension $0 \le k \le n-1$.  Previously, it was shown that on $\Gamma$, the free boundary $\partial \{u>0\}$ may be decomposed into a rectifiable set $\mathcal{S}$, which satisfies effective estimates, and a cusp set $\Sigma$. In this note, we prove that under mild assumptions, in the case $n = 2$ and $\Gamma$ a line, the cusp set $\Sigma$ does not exist. Building upon the work of Arama and Leoni \cite{AramaLeoni12}, our results apply to the physical case of a variational formulation of the Stokes' wave and provide a complete characterization of the singular portion of the free boundary in complete generality in this context.
\end{abstract}

\maketitle

\tableofcontents

\renewcommand{\thepart}{\Roman{part}}

\section{Introduction}\label{sec:intro}

In this paper, we provide very general sufficient conditions for eliminating cusps from the free boundary $\partial \{u>0\}$ of local minimizers $u$ for a class of \textit{degenerate} Alt-Caffarelli functionals
\begin{align}\label{alt-caffarelli functional}
J_Q(u, \Omega) := \int_{\Omega} |\nabla u |^2 + Q^2(x)\chi_{\{u > 0\}} dx,
\end{align}
so-called because we allow $Q = 0.$  These assumptions are satisfied by variational models of the Stokes Wave studied in \cite{AramaLeoni12,GravinaLeoni18,GravinaLeoni19}. While \cite{AramaLeoni12,GravinaLeoni18,GravinaLeoni19} eliminate cusps by strong assumptions of symmetry, this paper proves sufficient conditions for the non-existence of cusps without any assumption of symmetry.

A function $u$ is a \textit{minimizer} of (\ref{alt-caffarelli functional})
in the class $$K_{u_0, \Omega}: = \{u \in W^{1, 2}(\Omega) : u-u_0 \in W^{1,2}_0(\Omega) \}$$ for a $u_0 \in W^{1, 2}(\Omega)$ satisfying $u_0 \ge 0$ if for every function $v \in K_{u_0, \Omega},$ $J_{Q}(u, \Omega) \le J_{Q}(v, \Omega)$.  A function $u$ is called a \textit{local minimizer} of $J_Q(\cdot, \Omega)$ if there exists an $0<\epsilon_0$ such that $J_{Q}(u, B_r(x)) \le J_{Q}(v, B_r(x))$ for every $v \in K_{u, \Omega}$ satisfying
\begin{align}
\norm{\nabla (u-v)}^2_{L^2} + \norm{\chi_{\{u>0\}} - \chi_{\{v>0\}}}_{L^1(\Omega)} < \epsilon_0.
\end{align}

Since Alt and Caffarelli's watershed paper \cite{AltCaffarelli81}, assumptions on the function $Q$ play a crucial role in the study of the geometry of the \textit{free boundary} $\partial \{u>0\} \cap \Omega$ for a local minimizers $u$ of (\ref{alt-caffarelli functional}). For example, assuming that $Q$ is continuous and $0< Q_{\min} \le Q(x) \le Q_{\max} < \infty,$ near a free boundary point $x_0$ in $n=2$ dimensions, we can write 
$$u(x) = Q(x_0)(x-x_0, \vec \eta)_+ + O(|x - x_0|)$$
for some unit vector $\vec \eta$. Thus, if $0< Q(x_0)$ the blow-up of $u$ at $x_0$ is a piece-wise linear function. On the other hand, if $x \in \partial \{u >0\}$ and $Q(x) = 0$, $x$ cannot have piece-wise linear blow-ups and therefore $\partial \{u>0 \}$ must be singular at $x$.

The results of \cite{AltCaffarelli81} have inspired myriad generalizations and much interest (e.g., \cite{AltCaffarelliFriedman84}\cite{CRS10}  \cite{DeSilvaRoquejoffre12}\cite{DanielliPetrosyan05}\cite{Valdinoci04}\cite{DavidToro15}\cite{DavidEngelsteinGarciaToro19}). However, virtually all subsequent work on the geometry of the free boundary $\partial \{u>0\}$ has proceeded under the assumptions of the non-degenerate case, $0< Q_{\min} \le Q \le Q_{\max}<\infty$. Indeed, until recently the \textit{degenerate case} $0=Q_{\min} \le Q \le Q_{\max}<\infty$ has received little attention. One of the chief concerns in the degenerate case is the potential existence of cusps. Since the classical estimates in \cite{AltCaffarelli81} which are essential to regularity become vacuous in regions where $Q$ vanishes, establishing weak geometric results such as interior ball conditions becomes impossible using the standard techniques.

Previous work on the geometry of the free boundary in the degenerate case began by reformulating the theory of the Stokes wave, in which one searches for a solution to 
\begin{align} \label{e:free-boundary problem}\nonumber
\Delta u = 0 & \qquad \text{in  } \{u>0\} \subset \mathbb{R}^2\\
u = 0 & \qquad \text{on  } \partial \{u>0 \} \\
\nonumber \left|\frac{\partial}{\partial \eta} u(x, y)\right|^2 = -y  & \qquad \text{on  } \partial \{u>0\},
\end{align}
where $\eta$ is the outward normal to $\partial \{u>0\}$ in a variational setting.  Work on the Stokes Wave traditionally considers weak solutions to (\ref{e:free-boundary problem}), see \cite{WeissVarvaruca11} Definition 3.2 for details. While local minimizers of (\ref{alt-caffarelli functional}) are weak solutions of (\ref{e:free-boundary problem}), weak solutions of (\ref{e:free-boundary problem}) are critical points of (\ref{alt-caffarelli functional}) with some extra assumed regularity and may not be local minimizers of (\ref{alt-caffarelli functional}).  

Arama and Leoni \cite{AramaLeoni12} and subsequently Gravina and Leoni \cite{GravinaLeoni18} \cite{GravinaLeoni19} studied the Stokes Wave in the variational setting by considering local minimizers of (\ref{alt-caffarelli functional}) in $n=2$, $\Omega = \{(x, y): 0<x<1, 0<y \}$, and $Q(x, y) = \sqrt{(h-y)}_+$. They eliminate the possibility of cusps by assuming that the boundary data $u_0|_{\partial \Omega}$ is symmetric. Following their work, the author studied the degenerate case of local minimizers of (\ref{alt-caffarelli functional}) in broader generality \cite{McCurdy20}.  However, little was able to be said about cusps.  

\begin{lemma}(\cite{McCurdy20} Lemma 1.6 and Corollary 6.2)\label{l:finite perimeter}
Let $n \ge 2$, and $0 \le k \le n-1$. For any $0< \gamma$ and $k$-dimensional $C^{1, \alpha}$-submanifold $\Gamma$ if $Q(x)= \text{dist}(x, \Gamma)^\gamma$ and $u$ is a local minimizer of $J_Q(\cdot, B_2(0))$, then we can decompose $\partial \{u>0\} \cap \Gamma$ into $$\partial \{u>0\} \cap \Gamma = \Sigma \cup \mathcal{S}.$$ 
The set $\Sigma$ is the set of cusp points $x$ for which the density
\begin{align}\label{d:cusp density}
    \lim_{r \rightarrow 0^+}\frac{\mathcal{H}^n(B_r(x) \cap \{u>0\})}{\omega_n r^n} = 0.
\end{align}
The set $\mathcal{S}$ contains all the \textit{non-degenerate} singularities for which the density is strictly positive. Furthermore, $\mathcal{H}^{n-1}(\Sigma) = 0.$
\end{lemma}

In this paper, we restrict our attention to the physical case $n=2$ and $k = 1$ with $\Gamma$ flat, i.e., a line.  The main result of this paper is the following theorem.

\begin{theorem}(Main Theorem)\label{t:main theorem}
Let $n = 2$, $k =1$, $\Gamma = \{(x, 0): x \in \mathbb{R}\}$, and $0< \gamma$. Let $Q(x, y)= |y|^{\gamma}$. Let $u$ be a local minimizer of $J_Q(\cdot, B_2(0))$.  Suppose that $\{u>0\} \cap \Gamma = \emptyset$.  Then, $\Sigma = \emptyset.$
\end{theorem}

This theorem has a physical application to the Stokes wave. We note that \cite{WeissVarvaruca11}, in which the authors investigate weak solutions of (\ref{e:free-boundary problem}) without any assumptions of symmetry, contains the following sufficient conditions for eliminating cusps.

\begin{lemma}\label{Weiss Varvaruca no cusps}(\cite{WeissVarvaruca11}, Lemma 4.4)
Let $n = 2$, $k = 1$, $\Gamma$ is a linear subspace, and $Q(x) = \text{dist}(x, \Gamma)^{\frac{1}{2}}$. If $u$ is a weak solution of (\ref{e:free-boundary problem}) and $|\nabla u(x, y)| \le Q(x, y)$, $\Sigma = \emptyset.$
\end{lemma}
This lemma hold for local minimizers of (\ref{alt-caffarelli functional}).  Working in the context of local minimizers, we are able to prove a much stronger condition.

\begin{corollary}\label{c: 1}(Application to the variational formulation of the Stokes Wave)
Let $n=2$, $0< h <A$, $\Gamma = \{(x, h): x \in \mathbb{R}\}$, and $\Omega = [0, 1]\times[0, A]$.  Let $u_0 \in W^{1, 2}(\Omega)$ satisfy $u_0 = 0$ on $\partial \left([0,1] \times [h, A] \right)$ in the sense of traces. Let $Q(x, y) = \sqrt{|h-y|}$. 

Then, if $u$ is a local minimizer of $J_{Q}(\cdot, \Omega)$ in the class $K_{u_0, \Omega}$ and $\emph{supp}(u) \subset [0, 1] \times [0, h],$ then $\Sigma = \emptyset.$  

In particular, the results of \cite{McCurdy20} completely describe $\emph{sing}(\partial \{u>0\}).$ That is, for all compact sets $K \subset \Gamma$, $\partial \{u>0\} \cap \Gamma$ is locally finite and for all $x \in \partial \{u>0\} \cap \Gamma$ the unique blow-up (up to scaling and rotation) of $u$ at $x$ is given by 
\begin{align*}
u_{x,0^+} = r^{\frac{3}{2}}\sin(\theta\frac{3}{2}),
\end{align*}
defined on the sector $D$ for $D = \{(r, \theta): 0 \le \theta \le \frac{2\pi}{3} ; 0 \le r \}$.
\end{corollary}
The proof of Corollary \ref{c: 1} is an immediate application of Theorem \ref{t:main theorem}.  The proof of Theorem \ref{t:main theorem} makes essential use of minimality. We argue by contradiction, assuming that a cusp point $x \in \Sigma$ exists and building a competitor $v$ such that $J_{Q}(v, B_{2}(0)) < J_{Q}(u, B_2(0))$. Furthermore, we make essential use of the interior ball property in Lemma \ref{l:interior balls}, which only holds for local minimizers.

The intuition for the proof is that since $Q(x, y) \rightarrow 0$ as $y \rightarrow 0$, the submanifold $\Gamma$ should ``attract" the positivity set $\{u>0\}.$ Consider a function $u$ for which we can find a window $[0, N]\times [0, 3]$ in which $\{u>0 \}$ is a strip
\begin{align*}
\{u>0\} \cap \left([0, N]\times [0, 3]\right) = [0, N] \times [1, 2].
\end{align*}
Then we would expect $u$ to \emph{not} be a local minimizer, since $\{u>0\}$ should ``sag" under the attraction of $\Gamma$ for large $N$.

The strategy of the proof is to make this intuition rigorous.  That is, rather than analyze the behavior at a point $x \in \Sigma$, we consider a neighborhood \textit{away from} $x \in \Sigma$ in which we can perturb $u$ to be closer to $\Gamma$ and obtain the desired contradiction.  The main estimates rely upon Lemma \ref{l:increase estimate} and Lemma \ref{l:decrease estimate}.  The perturbation competitor itself is defined in Section \ref{S:competitor}. 

We note that the assumption $\{u>0\} \cap \Gamma = \emptyset$ is used twice in the proof: in Lemma \ref{l: e-nbhd} and Lemma \ref{l:decrease estimate}. In Lemma \ref{l: e-nbhd}, the assumption allows us to find a neighborhood of a component of $\{u>0\}$ so that we can perturb that component without pushing it into another component.  This is essential to define a competitor function.  In Lemma \ref{l:decrease estimate}, the assumption ensures that we can estimate the decrease in $(\ref{alt-caffarelli functional})$ from the resulting perturbation.

It is conjectured that the approach developed in this paper is extendable to higher dimensions and co-dimensions.

\subsection{Acknowledgements}
The author acknowledges the Center for Nonlinear Analysis at Carnegie Mellon University for its support.  Furthermore, the author thanks Giovanni Leoni and Irene Fonseca for their invaluable generosity, patience, and guidance.  

\section{Preliminaries}

Throughout this note, we fix $n=2$, $\Gamma = \{(x, 0): x \in \mathbb{R}\}$ and fix $0< \gamma.$  We set
\begin{align}
    Q(x, y):= |y|^\gamma.
\end{align}

We begin with a brief overview of some of the relevant results from \cite{AltCaffarelli81} and \cite{McCurdy20}.

\begin{definition}(Rescalings)
Let $u$ be a $\epsilon_0$-local minimizer of $(\ref{alt-caffarelli functional})$ in the $B_2(0)$, and let $x \in \Gamma \cap \partial \{u>0\}$.  We define the rescalings
\begin{align*}
u_{x, r}(y) := \frac{u(ry + x)}{r^{\gamma +1}}.
\end{align*}
Similarly, for set $A \subset \mathbb{R}^2,$ we define the rescalings
\begin{align*}
    A_{x, r} := \frac{1}{r}(A-x).
\end{align*}
\end{definition}

\begin{remark}\label{minimality preserved}(\cite{AltCaffarelli81} Remark 3.1)
Let $0< \gamma$.  Suppose that $u$ is an $\epsilon_0$-local minimizer of (\ref{alt-caffarelli functional}) in $B_2(0)$.  Then, if $x \in \Gamma \cap \partial \{u>0\},$ for any $0< r<1$, the function $u_{x, r}$ is an $\epsilon'$-local minimizer of (\ref{alt-caffarelli functional}) in the class $K_{u_{x, r}, B_{\frac{1}{r}}(0)},$ where
\begin{align*}
    \epsilon'= \epsilon_0 \max\{r^{-n}, r^{-(n-2\gamma)}\}.
\end{align*}
\end{remark}

\begin{remark}
The techniques used in \cite{AltCaffarelli81} to establish the non-degeneracy of a local minimizer $u$ rely upon comparing $u$ with two other functions:
\begin{enumerate}
    \item The harmonic extension of $u$ in a ball $B_r(0).$
    \item The function $w = \min\{u, v\}$ in $B_r(0)$ for
\begin{align*}
v(x) = \left(\sup_{y \in B_{r \sqrt{s}}(0)} \{u(y)\}\right) \max\left\{1 - \frac{|x|^{2-n} - r^{2-n}}{(sr)^{2-n} - r^{2-n}} , 0\right\}.
\end{align*} 
\end{enumerate}
Since $\norm{\nabla v}^2_{L^2} \le C(s, n) \sup_{y \in B_{r \sqrt{s}}(0)} \{u^2(y)\} r^{n-2}$ and harmonic functions are energy minimizers, for every $\epsilon_0$-local minimizer $u$, there is a uniform scale 
\begin{align} \label{e:standard scale}
    r_0:= r_0(n, \sup_{\partial \Omega} u_0, \norm{\nabla u}_{L^2}, \epsilon_0)
\end{align}
at which we can apply these arguments. We shall refer to this scale $r_0$ as the \textit{standard scale}.
\end{remark}

\begin{remark}
Let $u$ be an $\epsilon_0$-local minimizer of (\ref{alt-caffarelli functional}).  Suppose that $(0,0) \in \Gamma \cap \partial \{u>0\}.$ There is a rescaling of $u$ at $(0,0)$ as $u_{(0,0),r}$, where $0< r(n, \sup_{\partial B_2(0)}u_0, ||\nabla u||_{L^2(B_2(0))}, \epsilon_0)$ is such that $u_{(0,0),r}$ is a $1$-local minimizer in $\mathcal{K}_{u_{(0,0),r}, B_2(0)}$ and the standard scale $r_0$ for $u_{(0,0),r}$ is $r_0 = 1.$
\end{remark}

\begin{lemma}\label{c:local Lipschitz bound}(Local Lipschitz, \cite{McCurdy20} Corollary 3.14)
Let $u$ be a $1$-local minimizer of $J_{Q}(\cdot, B_2(0))$ with standard scale $r_0 = 1.$  Assume that $(x,y) \in \partial \{u>0\}$.  Then, for all  $(x',y') \in \{u>0\} \cap B_1((x, y))$ 
\begin{align*}
|\nabla u(x', y')| \le C(n)\max\{\dist((x', y'), \partial \{u>0\}), |y'|\}^{\gamma}.
\end{align*}
\end{lemma}

\begin{lemma}\label{l:interior balls}(Interior Balls, \cite{McCurdy20} Lemma 3.17)
Let $u$ be a $1$-local minimizer of $J_{Q}(\cdot, B_2(0))$ with standard scale $r_0 = 1.$  For any $B_R(y)$, we denote $Q_{\min, B_R(y)} = \min \{Q(z): z \in B_R(y) \}.$

Let $x \in \partial \{u>0 \}$ satisfy $B_{r}(x) \subset B_1(0) \setminus \Gamma$, then there exists a point $y \in \{u>0 \} \cap \partial B_{\frac{1}{2}r}(x)$ and a constant $0< c(n, Q_{\min, B_r(x)}) <\frac{1}{2}$ such that for all $z \in B_{\frac{c}{2}r}(y)$
\begin{align*}
    u(z) \ge C(n) r Q_{\min, B_{\frac{1}{2}r}(x)}.
\end{align*}
\end{lemma}

\begin{theorem}(\cite{McCurdy20} Theorem 4.3, Corollary 4.6
)\label{t:Weiss almost monotonicity}
For $u$ a local minimizer of $J_{Q}(\cdot, B_2(0))$.  Assume that $x_0 \in \Gamma \cap \partial \{u>0\}$. We define the Weiss $(1+ \gamma)$-density
\begin{align*}
W_{\gamma + 1}(x_0, r, u, \Gamma) := \frac{1}{r^{n-2+2(\gamma +1)}} \int_{B_r(x_0)}|\nabla u|^2 + Q^2(x)\chi_{\{u>0\}}dx - \frac{\gamma + 1}{r^{n-1+2(\gamma + 1)}}\int_{\partial B_r(x_0)} u^2 d\sigma
\end{align*}
For almost every $0<r\le 1 $
\begin{align*}
\frac{d}{dr}W_{\gamma +1}((0,0), r, u, \Gamma) & \le \frac{2}{r^{n+2\gamma}}\int_{\partial B_r((0,0))}\left(\nabla u \cdot \eta - \frac{\gamma +1}{r} u \right)^2d\sigma\\
W_{\gamma +1}((0,0), r, u, \Gamma) & \le W_{\gamma +1}((0,0), R, u, \Gamma).
\end{align*}
Furthermore, $\lim_{r \rightarrow 0^+} W_{\gamma +1}((0,0), r, u, \Gamma) = W_{\gamma +1}((0,0), 0^+, u, \Gamma)$ exists and $W_{\gamma + 1}((0,0), 0^+, u, \Gamma) \in [-c(n,\gamma, \alpha), c(n, \gamma, \alpha)]$.
\end{theorem}

\begin{theorem}\label{t:compactness}(Compactness,  \cite{McCurdy20} Theorem 5.3)
Let $u^i$ be a sequence of $1$-local minimizer of $J_{Q}(\cdot, B_2(0))$ with standard scale $r_0 = 1.$ Assume that $x_i \in \Gamma \cap \partial \{u>0\}.$ For any sequence of $r_i \rightarrow r \in [0, 1]$.  Let $B= B_\frac{1}{r}(0)$ if $r>0$ and $B = \mathbb{R}^n$ if $r = 0$. 

Then, there is a subsequence, $r_j \rightarrow r$ and a function $u \in C^{0}_{loc}(B) \cap W^{1, 2}_{loc}(B)$ such that
\begin{enumerate}
\item $u^j_{x_j, r_j} \rightarrow u$ in $C^{0}_{loc}(B) \cap W^{1, 2}_{loc}(B)$.
\item For every $0<R< \diam(B)$ and any $\epsilon >0$,  there is an $N \in \mathbb{N}$ such that for all $j \ge N$,
\begin{align*}
\partial \{u > 0 \} \cap B_r(0) \subset B_{\epsilon}(\partial \{u^j_{x_j, r_j} > 0 \}).
\end{align*}
\item For any $0< R< \diam(B)$ and any $\epsilon > 0$, there is an $N \in \mathbb{N}$ such that for all $j \ge N$,  
\begin{align*}
\partial \{u^j_{x_j, r_j} > 0 \} \cap B_R(0) \cap \{u > 0 \} \subset B_{\epsilon}(\partial \{u > 0 \}).
\end{align*}
Similarly, for any $0< R< \diam(B)$ and any $\epsilon > 0$, there is an $N \in \mathbb{N}$ such that for all $j \ge N$,  
\begin{align*}
\partial \{u^j_{x_j, r_j} > 0 \} \cap B_R(0) \cap \{u = 0 \} \subset B_{\epsilon}(\partial \{u > 0 \} \cup \Gamma).
\end{align*}
\item $\chi_{\{u^j_{x_j, r_j} >0 \}} \rightarrow \chi_{\{u>0 \}}$ in $L^1_{loc}(B)$.
\item The function $u$ is harmonic in $\{u > 0\}$.
\item For any $0<r\le1$ and any sequence of points $y_j \rightarrow y \in \Gamma \cap B_1(0)$
\begin{align*}
    W_{\gamma +1}(y_j, r, u_{x_j, r_j}, \Gamma) \rightarrow W_{\gamma +1}(y, r, u, \Gamma).
\end{align*}
\end{enumerate}
\end{theorem}

\begin{lemma}\label{l:Weiss density}(\cite{McCurdy20}Lemma 6.1)
Let $u$ be a minimizer of $J_{Q}(\cdot , \Omega)$ in the class $\mathcal{K}_{g, \Omega}.$  Let $x_0 \in \Gamma \cap \partial \{u>0 \}$.  Then
\begin{align*}
W_{\gamma +1}(x_0, 0^+, u, \Gamma) = \lim_{r \rightarrow 0^+}\frac{1}{r^{n + 2\gamma}}\int_{B_r(x_0)}Q^{2}(x)\chi_{\{u>0 \}}dx.
\end{align*}
Therefore, $W_{\gamma +1}(x_0, 0^+, u, \Gamma) \in [0, c_n],$ where $c_n = \int_{B_1(0)}|y|^{2\gamma}d(x, y).$ 

In particular, if $x_0 \in \Sigma$, then $W_{\gamma +1}(x_0, 0^+, u, \Gamma) =0$ and every blow-up $u_{x_0, 0}\equiv 0.$ Furthermore, the function $x \mapsto W_{\gamma +1}(x, 0, u, \Gamma)$ is upper semicontinuous when restricted to $\Gamma \cap \partial \{u>0 \} \cap B_1(0).$
\end{lemma}

\section{Local Behavior}

In this Section, we study the local geometry of $\{u>0 \}$ near $(0,0) \in \Sigma.$  The main result of this seciton is Lemma \ref{l:technical heart}, in which we obtain the neighborhoods in which we will work.  We note that the only lemma in this section which is motivated by the assumption that $\{u>0\} \cap \Gamma \not = \emptyset$ is Lemma \ref{l: e-nbhd}, in which we obtain a modicum of separation locally of the components of $\{u>0\}$.  This will be crucial later so that we may define a perturbation.

\begin{remark}(Attenuation Radius)\label{r:attenuation radius}
Let $n, k \in \mathbb{Z}$ such that $n \ge 2$ and $0\le k \le n-1$.  For $(x, y) \in \mathbb{R}^{k}\times \mathbb{R}^{n-k}$, we define $Q(x, y) = |y|^k$.  Suppose that $u$ is an $1$-local minimizer of (\ref{alt-caffarelli functional}) in $B_2(0) \subset \mathbb{R}^n$ for $Q$, as above, with standard radius $r_0 = 1$. Let $(0,0) \in \Sigma.$ By Lemma \ref{l:Weiss density} and Lemma \ref{l:interior balls} for any $0<\eta$ we can find a radius $0< r(\eta)$ such that 
\begin{align}
    \{u_{(0,0), r}>0\} \subset B_1(0) \subset \{(x, y) \in \mathbb{R}^k \times \mathbb{R}^{n-k}: |y| \le \eta |x|\}. 
\end{align}
\end{remark}

\begin{lemma}
Let $n=2$, $0< \gamma$, and let $u$ be an $1$-local minimizer of (\ref{alt-caffarelli functional}) in $B_2(0)$  in $\mathcal{K}_{u, B_2(0)}$ for $Q(x, y)= |y|^\gamma$ with standard radius $r_0 = 1$. $\Sigma$ is locally isolated.
\end{lemma}

\begin{proof}
We argue this by contradiction.  Without loss of generality, by translation and rescaling, we may assume that $(0, 0)  \in \Sigma$, that $(0, 0)$ is also a limit point of $\Sigma,$ and that $B_1(0) \subset \Omega.$ Furthermore, we may assume that the standard radius $r_0 = 1$ and that $u$ is a $1$-local minimizer.

Let $\mathcal{O}$ be a component of $\{u>0\}$ such that $(0, 0) \in \overline{\mathcal{O}}.$ Because $(0,0)$ is a limit point of $\Sigma$, either $\{(x, 0): x>0 \}$ contains infinitely many points of $\Sigma$ or $\{(x, 0): x<0\}$ does. Without loss of generality, we assume that the former case holds.  Similarly, $(0, 0)$ is either a limit point of $\{u>0\} \cap \{{x, y}: y>0 \}$ or $\{u>0\} \cap \{{x, y}: y<0\}$.  Without loss of generality, we assume that the former case holds.

By Remark \ref{r:attenuation radius}, for any $0< \eta$ we may find a radius $0<r(\eta)$ such that for all $0<r<r(\eta)$
$$
\{u>0\} \cap B_r(0, 0) \subset \{(x, y) \in \mathbb{R}^2: |y| \le \eta |x| \}.
$$
Let $0 < \eta \le \frac{1}{4}$ be fixed and let $0<r< r(\eta).$ 

\textbf{Claim 1:} We claim that for any $0<r<r(\eta),$ we can always find a point $(x, 0) \in \Sigma \cap B_{r}((0,0))$ such that $0<x<r$ and 
\begin{align*}
    \{(x, y): |y| \le \eta |x|\} \cap \{u=0\} \not = \emptyset.
\end{align*}
To prove the claim, we observe that for each $0< x_1 \le r$ such that $(x_1, 0) \in \Sigma \cap B_{r}((0,0))$, then $(x_1, 0)$ is the limit point of a component of $\{u>0 \} \cap B_r((0,0))$ which must intersect $\partial B_{r}((0,0))$ at either
\begin{align*}
    \partial B_r((0, 0)) \cap \{(x, y): 0<x,  |y| \le \eta x\}
\end{align*}
or
\begin{align*}
    \partial B_r((0, 0)) \cap \{(x,y): x<0, |y| \le \eta x\}.
\end{align*}
Since we are assuming, without loss of generality, that there is a sequence of $0< x_i$ such that $(x_i, 0) \in \Sigma$ and $x_i \rightarrow 0$, let $0< x_1<  x_2< x_3 \le r$ be three such points. If the component of $\{u>0 \} \cap B_r((0,0))$ which touches $(x_2, 0)$ intersects $\partial B_{r}((0,0))$ in $\partial B_r((0, 0)) \cap \{(x, y): 0<x,  |y| \le \eta x\}$, then $(x_3,0)$ satisfies the claim. If the component of $\{u>0 \} \cap B_r((0,0))$ which touches $(x_2, 0)$ intersects $\partial B_{r}((0,0))$ in $\partial B_r((0, 0)) \cap \{(x, y): x<0,  |y| \le \eta |x|\}$, then $(x_1,0)$ satisfies the claim.

Now, let $(x, 0) \in \Sigma \cap B_r((0,0)).$  Let $r_{x_1} := \sup\{y \in [0, \eta x]: (x,y) \in \{u>0\}\}.$

\textbf{Claim 2:} We claim that there is a constant, $\epsilon_2>0$, such that $W_{\gamma +1}((x_1, 0), 2r_{x_1}, u, \Gamma) > \epsilon_2$.  Note that by Lemma \ref{l:interior balls}, there is a ball of radius $c2r_{x_1}$ completely contained in $\{u>0\} \cap B_{2r_{x_1}}(x_1, 0).$ We verify the claim by a limit-compactness argument.  Suppose that there were a sequence of functions $u_i$ which have points $(x_i, 0) \in \Sigma$ and radii $r_i>0$ such that there is a ball of radius $cr_i$ contained in $\{u>0\} \cap B_{r_i}(x_i)$, for a fixed $0<c.$  And, assume that $W_{\gamma +1}((x_i, 0), r_i, u_i, \Gamma) \le 2^{-i}.$  By applying Lemma \ref{t:compactness} to $u^i_{x_i, r_i}$, we may pass to a subsequence which converges strongly to a function $u_{\infty}$ for which $W_{\gamma +1}((0, 0), 1, u_{\infty}, \Gamma) = 0$.  

Since $W_{\gamma + 1}((0, 0), 1, u_{\infty}, \Gamma)$ is monotonic non-decreasing, $$W_{\gamma +1}((0, 0), 0, u_{\infty}, \Gamma) = \Theta^{2}((0, 0), \{u_{\infty}>0\}) = 0.$$ By Theorem \ref{t:Weiss almost monotonicity} $u_{\infty}$ is homogeneous, which implies $u_{\infty} \equiv 0$. However, by Lemma \ref{l:interior balls} the compactness of the unit ball, and strong convergence, there exists a ball of radius $\frac{1}{2}c$ contained in $B_1(0) \cap \{u_{\infty} > 0\}.$ This contradiction proves the claim.

By the monotonicity of the Weiss density we see that $W_{\gamma +1}((x_1, 0), r, u, \Gamma) \ge \epsilon_2$ for all $r \ge 2r_{x}$. Since $2r_x \le 2\eta x$, for any fixed $0<r_0,$ we may repeat this calculation for any sequence $\{ x_i\} \subset \Sigma$ such that $\lim_{i \rightarrow \infty}x_i = x_0.$  For $x_i$ sufficiently close to $x_0,$ $W_{\gamma +1}((x_i, 0), r, u, \Gamma) \ge \epsilon_2$. Applying Lemma \ref{t:compactness} to $u_{x_i, r}$ implies that $W_{\gamma+1}((0, 0), r, u, \Gamma) \ge \epsilon_2$.  Since $0< r_0$ was arbitrary, this contradicts that assumption that $W_{\gamma +1}((0, 0), 0, u, \Gamma)=0$.  Therefore, $\Sigma$ is isolated.
\end{proof}

\begin{lemma}\label{l: e-nbhd}
Let $n=2$, $0< \gamma$, and let $u$ be an $1$-local minimizer of (\ref{alt-caffarelli functional}) in $B_2(0)$  in $\mathcal{K}_{u, B_2(0)}$ for $Q(x, y)= |y|^\gamma$ with standard radius $r_0 = 1$.  Let $R:= [a, b] \times [c, d]$ be a rectangle disjoint from $\Gamma$. Suppose that
\begin{align*}
    \{u>0\} \cap \partial R \subset \left(\{a\} \times [c, d]\right) \cup \left(\{b\} \times [c, d]\right).
\end{align*}
Then, for each component $\mathcal{O}_i \subset \{u>0 \} \cap R$ and $0<\eta$ there exists a $0<\epsilon(\eta, \mathcal{O}_i)$ such that $B_{\epsilon}(\mathcal{O}_i) \cap R_{\eta} \cap \mathcal{O}_j = \emptyset$ for all $j \not = i$, where we define $R_{\eta} := [a+\eta, b-\eta] \times [c, d]$. 
\end{lemma}

\begin{proof}
We argue by contradiction.  Suppose that there is a rectangle $R$ in which $\{u>0 \} \cap R$ consists of infinitely many components $\{\mathcal{O}_i\}_i$.  Now, by choosing a point in $x_i \in \partial \mathcal{O}_i$, we may extract a convergent subsequence $x_j \rightarrow x_{\infty} \in R$. Since $\partial \{u>0 \}$ is closed, $x_\infty \in \partial \{u>0\} \cap R$.

Now, we claim that all subsequential limits $x_{\infty} \in \partial R$.  Suppose that $B_{\epsilon}(x_{\infty}) \subset R.$ Then by \cite{AltCaffarelli81} Theorem 4.5 (3), 
\begin{align*}
    \mathcal{H}^{1}(B_\epsilon (x_\infty) \cap \partial \{u>0\}) < Cr^{1}.
\end{align*}
Since by assumption $x_j \rightarrow x_{\infty}$ and each $x_j$ is assumed to belong to a different component, each $x_j \in B_{\frac{1}{2}\epsilon}(x_{\infty})$ contributes at least $\epsilon$ length of $\partial \{u>0\}.$ This is a contradiction.

Thus, for a connected component $\mathcal{O}_i \subset \{u>0 \} \cap R$, and every every point $x \in \partial \mathcal{O}_i \setminus \partial R$, there is a radius $0<r_x \le \text{dist}(x, \partial R)$ such that $r_x$ is the maximal radius which satisfies 
\begin{align*}
    B_{r_x}(x) \cap \{u>0\} \subset \mathcal{O}_i.
\end{align*}
Furthermore, by the triangle inequality this radius $r_x$ is a continuous function of $x \in \partial \mathcal{O}_i \setminus \partial R$.  Therefore, for any $0<\eta$ $\partial \mathcal{O}_i \cap \overline{R_{\eta}}$ is compact and there exists an $0< \epsilon$ such that $0<\epsilon<r_x$ for all $x \in \partial \mathcal{O}_i \cap \overline{R_{\eta}}.$
\end{proof}

Now that we have an idea of the local geometry around points in $\Sigma,$ we show that we can find very specific windows, in which we will work.  This lemma works in higher-dimensions.

\begin{lemma}\label{l:technical heart}
Let $n,k \in \mathbb{z}$, $n \ge 2$, and $0 \le k \le n-1$.  Let $(x, y) \in \mathbb{R}^{k}\times \mathbb{R}^{n-k}.$  Let $0< \gamma$ and $Q(x, y) = |y|^\gamma$.  Let $u$ be an $1$-local minimizer of (\ref{alt-caffarelli functional}) for $Q$ in $B_2(0)$ with standard radius $r_0 = 1.$  Suppose that $(0, 0) \in \Sigma$.

For any $1<N< \infty$, we may find a radius $0< \rho(N)$ such that the rescaling $u_{(0, 0), \rho}$ satisfies the following conditions
\begin{itemize}
    \item [i.] There exists a large radius $1<<N_0<\infty$ satisfying $4N<N_0$ such that,
\begin{align}\label{e: standard window upper bound}
    \max \{|y| : (x, y) \in \left(\partial B^k_{N_0}(0) \times [-2,2]^{n-k}\right) \cap \partial \{u_{(0, 0), r} >0\}\} &= 1.
\end{align}
    \item [ii.] For the radius $N_0 - N$
\begin{align}
    \max \{|y| : (x, y) \in \left(\partial B^k_{N_0-N}(0) \times [-2,2]^{n-k}\right) \cap \partial \{u_{(0, 0), \rho} >0\}\} &\ge \frac{1}{2}.
\end{align}
\end{itemize}
That is, we may find two points $w, v \in \partial \{u_{(0, 0), \rho}>0 \}$ such that $w = (x_w, y_w)$ satisfies 
\begin{itemize}
    \item [a.] $|x_w| = N_0$
    \item [b.] $|y_w|  = 1 = \max \{y : (x, y) \in \left(\partial B^k_{N_0}(0) \times [-2,2]^{n-k}\right) \cap \partial \{u_{(0, 0), \rho} >0\}\}$
\end{itemize}
and $v= (x_v, y_v)$ satisfies 
\begin{itemize}
    \item [a.] $|x_v| = N_0 - N$
    \item [b.] $|y_v| = \max \{y : (x, y) \in \left(\partial B^k_{N_0-N}(0) \times \mathbb{R}^{n-k}\right) \cap \partial \{u_{(0, 0), \rho} >0\}\} \ge \frac{1}{2}.$
\end{itemize}
and, $\text{dist}(w, v) \ge N$.
\end{lemma}

\begin{proof} 
Let $1<< N < \infty$ be given.  Let $0<r= r(\frac{1}{4N})$ be the attenuation radius guaranteed by Remark \ref{r:attenuation radius}. We consider $u_{x, r}$ in the truncated cone
\begin{align}
    B_{1}(0) \cap \{(x, y) \in \mathbb{R}^k \times \mathbb{R}^{n-k}: |y| \le \frac{1}{4N} |x|  \}.
\end{align}
Let $z_0 = (x_0, y_0) \in \partial B_1(0)^k\times [-1, 1]^{n-k} \cap \partial \{ u_{x,r} > 0\}$ be a point which realizes 
\begin{align*}
    |y_0| := \max \{ |y| : (x, y) \in \partial B_1(0)^k\times [-1, 1]^{n-k} \cap \partial \{ u_{(0,0),r} > 0\}\}.
\end{align*}
Note that by assumption, $|y_0| \le \frac{1}{4N}$.  

We now consider the radius $r_1 = 1 -\frac{1}{2}2^{-1}$ and investigate the set $\partial \{u_{(0,0), r}>0\} \cap \partial B_{r_1}(0)^k\times [-1, 1]^{n-k}.$  If 
\begin{align*}
   |y_1| := \max \{|y| \in \mathbb{R}^{n-k} : (x, y) \in \partial \{u_{(0,0), r} >0\} \cap \partial B_{r_1}(0)^k\times [-1, 1]^{n-k} \} \ge \frac{1}{2} \left(\frac{1}{4N}\right),
\end{align*}
then the rescaling $u_{(0,0), \frac{r}{y_0}}$ proves the lemma. Suppose, for the purposes of contradiction, that
\begin{align*}
    |y_1| = \max \{|y| \in \mathbb{R}^{n-k} : (x, y) \in \partial \{u_{x, r} >0\} \cap B_{r_1}(0)^k \times [-1, 1]^{n-k} \} < \frac{1}{2^1} \left(\frac{1}{4N}\right).
\end{align*}
We consider the radius $r_2= r_1 - \frac{1}{2}2^{-2}$ and
\begin{align*}
    |y_2| = \max \{|y| \in \mathbb{R}^{n-k} : (x, y) \in \partial \{u_{x, r} >0\} \cap B_{r_2}(0)^k \times [-1, 1]^{n-k} \}.
\end{align*}
If $|y_2| \ge \frac{1}{2^2} \left(\frac{1}{4N}\right)$, then the rescaling $u_{x, \frac{r}{y_1}}$ satisfies the claim.  If $|y_2| < \frac{1}{2^2} \left(\frac{1}{4N}\right)$, then we proceed inductively.  If $r_i$ has been defined and $|y_i| < \frac{1}{2^i}\left(\frac{1}{4N}\right)$, then we consider $r_{i+1}= r_i - \frac{1}{2}2^{-i}$ and test 
\begin{align*}
|y_{i+1}| = \max \{|y| \in \mathbb{R}^{n-k} : (x, y) \in \partial \{u_{x, r} >0\} \cap B_{r_{i+1}}(0)^k \times [-1, 1]^{n-k}\},
\end{align*}
to see whether or not $|y_{i+1}| \ge \frac{1}{2^{i+1}}\left(\frac{1}{4N}\right).$  If for any $i \in \mathbb{N}$, $y_{i+1} \ge \frac{1}{2^{i+1}}\left(\frac{1}{4N}\right)$ then $u_{x, r\frac{1}{r_i}}$ satisfies the lemma. 

If, on the other hand, $|y_{i}| < \frac{1}{2^{i}}\left(\frac{1}{4N}\right)$ for all $i \in \mathbb{N}$, then we observe that
\begin{align*}
    \lim_{i \rightarrow \infty} r_i = \lim_{i \rightarrow \infty} 1 - \frac{1}{2}\sum_{k=1}^i\frac{1}{2^i} = \frac{1}{2}.
\end{align*}
We claim that this produces a contradiction.  Since $u_{(0,0), r} = 0$ on $\partial \{u_{(0,0), r}>0\}$ and $\Delta u_{(0,0), r} = 0$ in $\{u_{(0,0), r}>0\},$ we claim that if
\begin{align*}
    \max\{|y|: (x, y) \in \partial \{u_{(0,0), r}>0\} \cap \left(\partial B_{\frac{1}{2}}(0)^k \times [-1, 1]^{n-k}\right) \} = 0,
\end{align*}
then $u_{(0,0), r} = 0$ in $B_{\frac{1}{2}}(0)^k \times [-1, 1]^{n-k}$ by the Maximum Principle. Since this is a contradiction, there must be an $i \in \mathbb{N}$ such that $i$ is the first integer such that $|y_{i+1}| \ge \frac{1}{2^{i+1}}\left( \frac{1}{4N}\right).$ Thus, $u_{(0,0), r\frac{1}{r_i}}$ satisfies the lemma.
\end{proof}

\begin{remark}\label{r:find window}
We may repeat this argument, restricting our attention to any connected component $\mathcal{O}$ of $\{u>0\}$ which touches $(0, 0) \in \Sigma$ to obtain the same result. 
\end{remark}

In the next lemma, we find height control on components $\mathcal{O}$ for which $(0,0) \in \overline{\mathcal{O}}$.

\begin{lemma}\label{l:height bound}(Height bound)
Let $u$ is an $1$-local minimizer of (\ref{alt-caffarelli functional}) with standard radius $r_0 = 1$.  Suppose that $(0,0) \in \Sigma$. Let $\mathcal{O}$ be any component of $\{u>0 \}$ such that $(0,0) \in \overline{\mathcal{O}}.$ Let $N \in \mathbb{N}$ be fixed.  Let $0< r(N)$ as in Lemma \ref{l:technical heart}, applied to $\mathcal{O}$ as in Remark \ref{r:find window}. There exist constants $0< C_1(\gamma)< C_2(\gamma)< \infty$ independent of $N$ such that the following holds. If we define $S := [N_0 -N, N_0] \times \mathbb{R}$, then
\begin{align}\nonumber
    C_1(\gamma) & < \inf_{x \in [N_0 - N + \frac{1}{4}, N_0]} \sup \{|y| : (x, y) \in \mathcal{O}_{(0,0), r}\cap S \} \\
    & \le \sup_{x \in [N_0 - N, N_0]} \sup\{|y| : (x, y) \in \mathcal{O}_{(0,0), r} \cap S \}< C_2(\gamma).
\end{align}
\end{lemma}

\begin{proof}
We begin by observing that for any $0<x_0 \le N_0$
\begin{align*}
    \max \{u_{(0,0), r}(x, y) : (x, y) \in \mathcal{O}_{(0,0), r} \text{  and  } 0 < x<x_0 \}
\end{align*}
must occur on $\{x_0 \} \times \mathbb{R} \cap \mathcal{O}$ by the Maximum Principle. Thus
\begin{align*}
    \max \{u_{(0,0), r}(x, y) : (x, y) \in \mathcal{O}_{(0,0), r} \text{  and  } 0 < x<N_0 \}
\end{align*}
must occur on $\{N_0 \} \times \mathbb{R} \cap \mathcal{O}$.  By the Lipschitz bound in Corollary \ref{c:local Lipschitz bound} and (\ref{e: standard window upper bound}), then
\begin{align*}
    \max \{u_{(0,0), r}(x, y) : (x, y) \in \mathcal{O}_{(0,0), r} \text{  and  } 0 < x< N_0 \} \le C(\gamma).
\end{align*}
We claim that this implies that there exists a $0<C_2(\gamma)<\infty$ such that 
\begin{align*}
    \sup_{0<x<N_0} \sup\{|y| : (x, y) \in \mathcal{O}_{(0,0), r} \cap S \} \le C_2(\gamma).
\end{align*}
Suppose that $(x, y) \in \partial \mathcal{O} \cap S$.  Then, by Lemma \ref{l:interior balls}, we can find a point $z \in B_{\frac{y}{2}}((x, y)) \cap \mathcal{O}$ such that,
\begin{align*}
    u(z) \ge C_{min} \left(\frac{|y|}{2}\right)^{1 + \gamma}.
\end{align*}
However, by our construction of $u_{(0,0), r}$ and Lemma \ref{c:local Lipschitz bound} 
\begin{align*}
    \sup_{0<x<N_0} \sup\{u_{(0,0), r}(x, y) : (x, y) \in \mathcal{O}_{(0,0), r} \cap S \} \le \sup\{u_{(0,0), r}(N_0, y) : (N_0, y) \in \mathcal{O}_{(0,0), r}\}  \le C(\gamma).
\end{align*}
Therefore, $|y| \le 2\left( \frac{C(\gamma)}{C_{min}}\right)^{\frac{1}{1 + \gamma}} = C_2(\gamma)$.

The other estimate follows from applying the same argument on the side $\{N_0 - N\} \times \mathbb{R}$.  Let $y_0$ be such that $|y_0| \ge |y|$ for all $y$ such that $(N_0 - N, y) \in \partial \mathcal{O}$.  We observe that $|y_0| \in [\frac{1}{2}, C_2(\gamma)].$ Therefore, by Lemma \ref{l:interior balls}
\begin{align*}
     \sup_{(x, y) \in \mathcal{O}} \{u(x, y): N_0 - N + \frac{1}{4} \le x \le N_0\} & \ge \sup_{(x, y) \in B_{\frac{1}{4}}(N_0 - N, y_0)} u(x, y)\\
     & \ge C_{min} \left(C_2(\gamma)/2\right)^{1 + \gamma}.
\end{align*}
Thus, by Corollary \ref{c:local Lipschitz bound} there must be a constant $0< C_1(\gamma)$ such that
\begin{align*}
    \inf_{N_0 - N + \frac{1}{4}<x<N_0} \sup\{|y| : (x, y) \in \mathcal{O} \cap S \} \ge C_1(\gamma).
\end{align*}
\end{proof}

\begin{definition}
For $0< \gamma$ fixed, and let $\mathcal{O}$ be any component of $\{ u>0\}$ such that $(0,0) \in \overline{\mathcal{O}}.$  We make the following definitions. If $0< \gamma < \frac{1}{2},$ then we define $u_{(0, 0), C_2(\gamma)2\gamma r}$ as the \textit{standard rescaling} at $(0,0)$ for scale $C_2(\gamma)^{-1}2\gamma N$. Correspondingly, we shall call the strip
\begin{align*}
    S := \left[C_2(\gamma)^{-1}2\gamma(N_0 - N), C_2(\gamma)^{-1}2\gamma N_0 \right] \times \mathbb{R}
\end{align*}
the \textit{standard window} at $(0,0)$ for scale $C_2(\gamma)^{-1}2\gamma N$.  

If $\frac{1}{2} \le \gamma,$ then we define $u_{(0, 0), C_2(\gamma)^{-1}r}$ as the \textit{standard rescaling} at $(0,0)$ for scale $C_2(\gamma)^{-1}N$. Correspondingly, we shall call the strip
\begin{align*}
   S:= \left[C_2(\gamma)^{-1}(N_0 - N), C_2(\gamma)^{-1} N_0 \right] \times \mathbb{R}
\end{align*}
the \textit{standard window} at $(0,0)$ for scale $C_2(\gamma)^{-1} N$.  

We note that while these rescalings and windows depend upon $\mathcal{O},$ they enjoy the following property. If denote 
\begin{align*}
\mathcal{O}' := \begin{cases}
    \mathcal{O}_{(0, 0), C_2(\gamma)^{-1}2\gamma r} & 0< \gamma < \frac{1}{2}\\
    \mathcal{O}_{(0, 0), C_2(\gamma)^{-1} r} &  \frac{1}{2} \le \gamma.
    \end{cases}
\end{align*}
as the \textit{standard component}, then
\begin{align*}
    \mathcal{O}' \cap S \subset \begin{cases}
    [C_2(\gamma)^{-1}2\gamma(N_0 - N), C_2(\gamma)^{-1}2\gamma N_0 ] \times [-2\gamma, 2\gamma] & 0< \gamma < \frac{1}{2}\\
    [C_2(\gamma)^{-1}(N_0 - N), C_2(\gamma)^{-1} N_0] \times [-1, 1] &  \frac{1}{2} \le \gamma.
    \end{cases}
\end{align*}
\end{definition}

\begin{corollary}\label{c:density}
Let $n=2$, $0< \gamma$, and let $u$ be an $1$-local minimizer of (\ref{alt-caffarelli functional}) in $B_2(0)$  in $\mathcal{K}_{u, B_2(0)}$ for $Q(x, y)= |y|^\gamma$ with standard radius $r_0 = 1$. Let $(0,0) \in \Sigma$, and $1<< N \in \mathbb{N}$. Let $v$ be the standard rescaling of $u$ at $(0,0)$ for scale $N$, $S$ the corresponding standard window, and $\mathcal{O}'$ the component with respect to which they are standard.

There exists a $c(\gamma)$ such that for every $x$ such that $[x-\frac{1}{2}, x+ \frac{1}{2}] \times \mathcal{R} \subset S$
\begin{align}
   \mathcal{H}^2\left( \mathcal{O} \cap \left([x- \frac{1}{2}, x + \frac{1}{2}] \times [0, 2] \right)\right) \ge \omega_2\left(c(2, \gamma)C_1(\gamma)\right)^2,
\end{align}
where $c(2, \gamma) = c\left(2, \frac{C_1(\gamma)^{ \gamma}}{2^{\gamma}}, \left(C_2(\gamma)+ \frac{C_1(\gamma)}{2}\right)^{\gamma}\right)$ from Lemma \ref{l:interior balls}.
\end{corollary}

\begin{proof}
By Lemma \ref{l:height bound}, there is a point $(x, y) \in \partial \mathcal{O} \cap \left([x- \frac{1}{2}, x + \frac{1}{2}] \times [0, 2]\right)$ such that $|y| \ge C(\gamma)$. Lemma \ref{l:interior balls} applied to the ball $B_{\frac{y}{2}}(x, y)$ proves the corollary.
\end{proof}

\section{Domain Perturbations}

In this section, we consider the basic properties of a simple domain perturbation. 

\begin{definition}
We define a family of functions, $F_{t}:\mathbb{R}^2 \rightarrow \mathbb{R}^2$, as follows.  For $x = (x_1,x_2) \in \mathbb{R}^1\times \mathbb{R}^{1}$
\begin{align}
    F_t(x_1, x_2) := (x_1, x_2+tx_1).
\end{align}
For $u$ a Lipschitz function, we define $u_t$ as 
\begin{align}
    u_t := u \circ F_{t} = u(x_1, x_2 + tx_1).
\end{align}
\end{definition}

\begin{lemma}\label{l:increase estimate}(Increase Estimate)
Let $u: \mathbb{R}^2 \rightarrow \mathbb{R}^2$ a Lipschitz function and $a, b, c, d \in \mathbb{R}$ satisfy $a < b$ and $0 \le c < d$.  Consider the rectangle $R = [0, b] \times [c, d]$.  Suppose that there exists an $0<\epsilon$ such that
\begin{align*}
    \text{supp}(u) \cap R \subset [a, b] \times [c+\epsilon, d - \epsilon].
\end{align*}
Then, for $0< t$ sufficiently small
\begin{align}
    \norm{\nabla u_t}^2_{L^2(R)} \le \norm{\nabla u}^2_{L^2(R)}(1 + t + t^2).
\end{align}
\end{lemma}

\begin{proof}
By assumption, for sufficiently small $0<t$, $F_t(\text{supp}(u) \cap R) \subset R.$  Thus, the following estimate gives the desired estimate.
\begin{align*}
    |\nabla u_t|^2 & = (\partial_1 u + t\partial_2 u)^2 + (\partial_2 u)^2 \\
    & =  (\partial_1 u)^2 + 2t\partial_1 u \partial_2 u + t^2 (\partial_2 u)^2 + (\partial_2 u)^2 \\
    & \le |\nabla u|^2 + t |\nabla u|^2 + t^2 |\nabla u|^2\\
    & \le |\nabla u|^2 (1 + t + t^2).
\end{align*}
\end{proof}

\begin{lemma}(Decrease Estimate)\label{l:decrease estimate}
Let $0< \gamma$ and let $Q(x, y) = |y|^{\gamma}$. Let $0< a, b, c, d$, with $a<b$ and $0< c<d$, and let $\Omega \subset [a, b] \times [c, d] \subset \mathbb{R}^2$ be an open set. Suppose the following conditions hold
\begin{itemize}
    \item [i.] If $\gamma \ge \frac{1}{2},$ then $d \le 1$.
    \item [ii.] If $0<\gamma < \frac{1}{2}$, then $d \le 2\gamma$. 
\end{itemize}
Then
\begin{align*}
   \frac{d}{dt}|_{t= 0} \int_{F_{-t}(\Omega)} Q^2(x)dx & \le -a \int_{\Omega} Q^2(x)dx.
\end{align*}
\end{lemma}

\begin{proof}
For each $x \in [a, b]$, we write $\Omega_x = \Omega \cap \left(\{x\} \times \mathbb{R}\right).$ Note that $\Omega_x$ is a relatively open set, and is therefore a countable union of open intervals. We note that
\begin{align*}
   \frac{d}{dt} \int_{F_{-t}(\Omega)} Q^2(x)dx & =  \int_a^b \frac{d}{dt} \int_{F_{-t}(\Omega_x)}y^{2\gamma}dy dx.
\end{align*}
Investigating a single interval $(y_1, y_2)$ we obtain
\begin{align*}
   \frac{d}{dt} \int_{y_1-tx}^{y_2-tx}y^{2\gamma}dy & = \frac{d}{dt} \frac{1}{2\gamma +1}((y_2-tx)^{2\gamma+1}- (y_1-tx)^{2\gamma+1})\\
    & = (-x)((y_2-tx)^{2\gamma}- (y_1-tx)^{2\gamma})\\
    & \le -a ((y_2-tx)^{2\gamma}- (y_1-tx)^{2\gamma}).
\end{align*}
Evaluating at $t = 0$, we obtain $\frac{d}{dt} \int_{y_1-tx}^{y_2-tx}y^{2\gamma}dy \le -a (y_2^{2\gamma}- y_1^{2\gamma})$.

By assumption $(i.)$, if $\gamma \ge \frac{1}{2},$ then  $y_1, y_2 \in [0, 1]$.  We note that $2\gamma y^{2\gamma-1} \ge y^{2\gamma}$ for $y \in [0, 1]$ and $\gamma \ge \frac{1}{2}.$  Thus
\begin{align*}
    (y_2^{2\gamma}- y_1^{2\gamma}) & = \int_{y_1}^{y_2} 2\gamma y^{2\gamma - 1}dy\\
    & \ge \int_{y_1}^{y_2} y^{2\gamma} dy\\
    & = \frac{1}{2\gamma + 1}(y_2^{2\gamma+1}-y_1^{2\gamma+1}).
\end{align*}

If $0<\gamma < \frac{1}{2}$, then by assumption $(ii.)$, $y_1, y_2 \in [0, 2\gamma]$. We note that for $y \in [0, 2\gamma]$, $2\gamma y^{2\gamma-1} \ge y^{2\gamma}$.  Thus, for $y_1, y_2 \in [0, 2\gamma]$, the previous estimate holds.

Therefore, we may sum over all the intervals $(y_{2i-1}, y_{2i}) \in \Omega_x$ and obtain
\begin{align*}
    \sum_{i} -a (y_{2i}^{2\gamma}- y_{2i-1}^{2\gamma}) \le -a \int_{\Omega_x}y^{2\gamma}dy.
\end{align*}
Integrating with respect to $x \in [a, b]$, we obtain 
\begin{align*}
    \frac{d}{dt}|_{t = 0} \int_{F_{-t}(\Omega)} Q^2(x)dx  & \le \int_a^b - a \int_{\Omega_x}y^{2\gamma}dy\\
    & \le -a\int_{\Omega}Q^2(x)dx.
\end{align*}
\end{proof}

\begin{remark}
Lemma \ref{l:decrease estimate} is the only place, so far, where the assumption that $\{u>0\} \cap \Gamma = \emptyset$ comes into play.  Indeed, under this assumption, then if $v$ is the standard rescaling of $u$ at $(0,0) \in \Sigma$ at scale $N$, then $\{v>0\}$ satisfies the hypotheses of Lemma \ref{l:decrease estimate} in the standard window.
\end{remark}

\section{Construction of the Competitor}\label{S:competitor}

Let $u$ be an $1$-local minimizer of the functional (\ref{alt-caffarelli functional}) with standard radius $r_0 = 1.$ Suppose that $(0,0) \in \Sigma$ and $\{u>0 \} \cap \Gamma = \emptyset$. Without loss of generality, we assume that $(0,0) \in \overline{\{u >0\} \cap \{(x, y): y>0\}}$. In this section, we define a competitor $v$ for $u$ in a small neighborhood of $x.$  

For any $1<< N < \infty$, we let $v$ be the standard scaling at $(0,0)$ at scale $2N+1$, $S= [a, b]\times [0, c]$ the standard window, and $\mathcal{O}'$ the standard component. We decompose $S$ into two special sub-regions
\begin{align*}
    S^L & := [a+\frac{1}{2}, \frac{b-a}{2}] \times [0, c]\\
    S^R & := [\frac{b-a}{2}, b-\frac{1}{2}] \times [0, c].
\end{align*}
Let $\tau_{x_0}:\mathbb{R}^2 \rightarrow \mathbb{R}^2$ be translation in the $x$-variable
\begin{align*}
    \tau_{x_0}(x, y) = (x-x_0, y).
\end{align*}

By Lemma \ref{l: e-nbhd}, we may find an $0<\epsilon$ such that $B_{\epsilon}(\mathcal{O}' \cap \left([a+\frac{1}{2}, b- \frac{1}{2}]\times[0, c] \right))$ does not intersect any other components of $\{v>0\} \cap S$.  Similarly, there is an $0< \eta$ such that $\dist(\partial \mathcal{O}' \cap S, \Gamma) \ge \eta.$  Let $\eta':= \min\{\eta, \epsilon\}$

For $0< t \le \epsilon' \frac{1}{2N}$, we define a family of functions $v_t: S
\rightarrow \mathbb{R}^2$ piece-wise as follows.
\begin{align*}
    v_t(x, y) & = \begin{cases}
     v \circ \tau_{-a-\frac{1}{2}} \circ F_t \circ \tau_{a+ \frac{1}{2}}(x, y) & (x, y) \in B_{\frac{\epsilon}{2}}(\mathcal{O}') \cap S^L \\
     v \circ \tau_{-b+\frac{1}{2}} \circ F_{-t} \circ \tau_{b-\frac{1}{2}}(x, y) & (x, y) \in B_{\frac{\epsilon}{2}}(\mathcal{O}') \cap  S^R\\
     v(x, y) & (x, y) \in S \setminus (B_{\frac{\epsilon}{2}}(\mathcal{O}') \cap (S^L \cup S^R))
    \end{cases}
\end{align*}
Note that $v_t \in K_{v, S}$ for $0< t \le \epsilon' \frac{1}{2N}$.

\section{Proof of the Main Theorem}

Since $v_t \in K_{v, S}$, by Remark \ref{minimality preserved}  $J_{Q}(v, S_1) \le J_{Q}(v_t, S)$ for all sufficiently small $0<t$. We shall define 
\begin{align*}
I_1(t)& := \int_{B_{\frac{\epsilon}{2}}(\mathcal{O}') \cap S}|\nabla v_t|^2dx\\
I_2(t)& := \int_{B_{\frac{\epsilon}{2}}(\mathcal{O}')\cap S}Q^2(x)\chi_{\{v_t>0\}}(x)dx\\
I(t) & := I_1(t) + I_2(t).
\end{align*}
Note that by the construction of $v_t$, since $v_t = v$ in $S \setminus (B_{\frac{\epsilon'}{2}}(\mathcal{O}') \cap (S^L \cup S^R))$.  Therefore, by the assumption of local minimality, for all $0< t< \epsilon \frac{1}{2N}$ we have $I(0) \le I(t).$  We shall prove the main theorem by showing that for all sufficiently small $0< t$, $I(0) > I(t).$  

By Lemma \ref{l:increase estimate}, and Corollary \ref{c:local Lipschitz bound}, we have that 
\begin{align*}
    I_1(t) & \le (1 + t + t^2)I_1(0)\\
    & \le I_1(0) + 2(t + t^2) NC^2(2,\gamma).
\end{align*}

To obtain the necessary decrease estimate, we decompose $S$ into $N-1$ regular subintervals of width $1$.  Then, Lemma \ref{l:decrease estimate}, Corollary \ref{c:density}, and Lemma \ref{l:interior balls} applied to each subinterval imply
\begin{align*}
    \frac{d}{dt}|_{t=0} I_2(t) & \le -2\sum_{i = 1}^{N-1}(i-1) C(c, \gamma) \left(\frac{C_1(\gamma)}{C_2(\gamma)}\right)^{2\gamma}\\
    &\le - c_4(2, \gamma)\sum_{i = 1}^{N-1}(i-1)\\
    &\le - c_4(2, \gamma)\frac{\left(N -2 \right) \left(N -3  \right)}{2}\\
    &\le - c_5(2, \gamma)(N^2 -5N +6).
\end{align*}
Note that for $N$ sufficiently large $c_5(2, \gamma)(N^2-5N +6) \ge 5NC^2(2, \gamma).$  Thus, for $N$ sufficiently large and $0< t$ sufficiently small
\begin{align*}
    I(t) - I(0) & \le - t\frac{1}{2}c_5(2, \gamma)(N^2-3N +3) + (t + t^2)2NC^2(2, \gamma)\\
    & \le - t\frac{1}{2}c_5(2, \gamma)(N^2-3N +3) + t4NC^2(2, \gamma)\\
    & \le -t\frac{1}{2}NC^2(2, \gamma)<0.
\end{align*}
This contradicts the minimality of $v.$  Hence, $\Sigma = \emptyset.$

\bibliography{references}
\bibliographystyle{amsalpha}

\end{document}